\documentclass[12pt]{amsart}
\usepackage{geometry}                
\usepackage{amsmath,amsfonts,amsthm,amssymb,amsopn,cite,mathrsfs}
\usepackage{epsfig,verbatim}
\usepackage{subfigure}
\usepackage{graphicx}
\usepackage{epstopdf}
\begin{document}

\newcommand{\beq}{\begin{equation}}
\newcommand{\eeq}{\end{equation}}
\newcommand{\beqs}{\begin{equation*}}
\newcommand{\eeqs}{\end{equation*}}
\newcommand{\beqq}{\begin{equation}}
\newcommand{\eeqq}{\end{equation}}

\newcommand{\beqas}{\begin{eqnarray*}}
\newcommand{\eeqas}{\end{eqnarray*}}
 \newcommand{\beqa}{\begin{eqnarray}}
\newcommand{\eeqa}{\end{eqnarray}}

\newcommand{\bmuls}{\begin{multline*}}
\newcommand{\emuls}{\end{multline*}}
 \newcommand{\bmul}{\begin{multline}}
\newcommand{\emul}{\end{multline}}

\newcommand{\bs}{\boldsymbol}
\newcommand{\mbf}{\mathbf}
 
 \newcommand{\xa}{\xi(\alpha)}
 \newcommand{\xb}{\xi(\beta)}
 \newcommand{\laa}{L(\alpha,\alpha)}
 \newcommand{\lab}{L(\alpha,\beta)}
 \newcommand{\lng}{\langle}
 \newcommand{\rng}{\rangle}
 \newcommand {\res}{\text{Res} \, }
\newcommand{\ola}{\overline{\Lambda^\circ}}

\def\bof{\mathbf {f}} 
\def\bog{\mathbf {g}} 
\def\boG{\mathbf {G}} 
\def\boh{\mathbf {h}}

 \def\bC{\mathbb{C}}
  \def\bD{\mathbb{D}}
 \def\bR{\mathbb{R}}
 \def\bT{\mathbb{T}}
   \def\bZ{\mathbb{Z} }

  \def\cA{\mathcal{A}} 
  \def\cB{\mathcal{B}}  
   \def\cE{\mathcal{E}}
\def\cF{\mathcal{F}}
\def\cG{\mathcal{G}}
\def\cH{\mathcal{H}}
\def\cO{\mathcal{O}}
\def\cP{\mathcal{P}}

\def\bga{\bs{\gamma}}

\def\BT {Bargmann transform \ }
\def\FS {Fock space \ }
\def\HF {Hermite function \  } 
\def\HFS {Hermite functions \  }
\def\TFS {time-frequency shift}
\def\SIF{$\sigma$-function \ }
\def\const{\text{Const}\ }
\def\dist{\text{dist}\ }

\newcommand{\vvlrd}{L^2(\bR , \bC ^n)}
\newcommand{\rd}{\bR ^d}
\newcommand{\rdd}{\bR ^{2d}}
\newcommand{\inv}{^{-1}}
\newcommand{\lrd}{L^2(\rd )}
\def\lr{L^2(\bR)}
\newcommand{\tfs}{time-frequency shift}
 \def\cS{\mathcal{S}}
\def\cK{\mathcal{K}}
\newcommand{\fif}{if and only if}

\newtheorem{thm}{Theorem}[section]
\newtheorem{lemma}[thm]{Lemma}
\newtheorem{defin}{Definition}[section]
\newtheorem{prop}[thm]{Proposition}
\newtheorem{cor}[thm]{Corollary}


\begin{abstract}
We investigate vector-valued  Gabor frames (sometimes called Gabor
superframes) based  on   Hermite 
functions $H_n$. Let $\boh = (H_0, H_1, \dots ,  H_n)$ be the vector of the
first $n+1$ Hermite functions.  We  give a complete characterization
of all lattices  $\Lambda \subseteq \bR ^2$  such that the Gabor system
 $\{ e^{2\pi i \lambda
_2 t} \boh (t-\lambda _1): \lambda = (\lambda _1, \lambda _2) \in
\Lambda \}$ is a frame for $L^2 (\bR , \bC ^{n+1} )$. As a corollary
we obtain sufficient conditions for a single Hermite function to
generate a Gabor frame and a new estimate for the lower frame
bound. 
The  main tools are  growth estimates for the Weierstrass
$\sigma $-function, a new type of interpolation problem  for entire
functions on the Bargmann-Fock space, and structural results about
vector-valued Gabor frames. 
\end{abstract}

\title{Gabor (Super)Frames with Hermite Functions}
\author{Karlheinz Gr\"ochenig}
\address{Faculty of Mathematics \\
University of Vienna \\
Nordbergstrasse 15 \\
A-1090 Vienna, Austria}
\email{karlheinz.groechenig@univie.ac.at}
\author{Yurii Lyubarskii}
\address{Department of Mathematics,
Norwegian University of Science and Technology,
7491, Trondheim, Norway}
\email{yura@math.ntnu.no}
\thanks{K.~G.~was supported by the Marie-Curie Excellence Grant
  MEXT-CT 2004-517154. Yu.L. was partially supported by the
Research Council of Norway grant 10323200.\\
This research is a part of the European Science  Foundation Networking Programme
"Harmonic   and Complex Analysis "}

\maketitle

\section{Introduction}

Given a function $g\in \lr$ and a lattice $\Lambda \subset \bR ^2$,
we study the frame property of the set $\{e^{2\pi i \lambda _2 t}
g(t-\lambda _1):  \lambda = (\lambda _1,\lambda _2) \in \Lambda\}$.
 Precisely,  write $\pi _\lambda g = e^{2\pi i \lambda _2 t}
g(t-\lambda _1)$ for the \tfs\ by $\lambda=(\lambda_1,\lambda_2) \in \bR ^2$. Then we call
the set $\cG (g,\Lambda ) = \{ \pi  _\lambda g : \lambda \in \Lambda
\}$ a \emph{Gabor frame} or \emph{Weyl-Heisenberg frame}, whenever
there  exist  constants $A,B >0$ such that, 
for all $f\in \lr$,
\begin{equation}
\label{Eq:2.8} A\|f\|^2_{\lr} \leq \sum_{\lambda \in
\Lambda}|\langle f, \pi_\lambda g\rangle_{\lr}|^2 \leq
B\|f\|^2_{\lr}\, .
\end{equation}
  Gabor frames  originate in quantum mechanics through J.~von
Neumann and in information theory through D.~Gabor~\cite{Ga}   and
nowadays have many applications in signal processing. A large body
of results describes the structure of Gabor frames and provides
sufficient conditions of a qualitative nature for  $\cG (g,\Lambda )$
to form a (Gabor) frame, see~\cite{daubechies90,Charlybook} for
details and references.

We will also study vector-valued Gabor frames. In this
case the Hilbert space is 
$\cH=L^2(\bR, \bC^n)$ consisting of all vector-valued functions 
$\bof(t)= \big( f_1(t), \ldots \ , f_n(t) \big)$ with the natural inner product
\beq
\label{product}
\langle \bof, \bog \rangle = \sum_{j=1}^n 
      \int_{-\infty}^\infty f_j(t) \overline{g_j(t)}dt=
               \sum_{j=1}^n \langle f_j, g_j \rangle .
\eeq               
 The {\em time-frequency shifts} $\pi _z$, $z= (x,\xi )$ act coordinate-wise by
\begin{equation}
  \label{eq:c1}
\pi _z\bof (t) =  e^{2i\pi \xi t}\bof(t-x) \, .
\end{equation}
The vector-valued Gabor system
$\cG(\bog, \Lambda)= \{ \pi_\lambda \bog : \lambda \in \Lambda\}   $
is a frame for $L^2(\bR , \bC ^n)$, if there exist constants $A,B>0$  such that 
\beq
\label{frameinequalitya}
A \|\bof\|^2_{L^2(\bR, \bC^n)} \leq \sum_{\lambda \in \Lambda}
 |\langle \bof , \pi_\lambda \bog \rangle |^2 \leq B\|\bof \|^2_{L^2(\bR, \bC^n)},
 \quad \forall \bof \in  L^2(\bR, \bC^n).
 \eeq

Vector-valued Gabor frames were introduced under the name
``superframes'' in signal processing by 
R.~Balan~\cite{balan1} in the context of ``multiplexing'' and studied in
~\cite{balan2,han-larsen} for their own sake.  The idea of
``multiplexing'' is to encode $n$ independent signals
(functions) $f_j\in L^2(\bR ), j=1,\dots , n,$ as a single sequence that captures the
time-frequency information of each $f_j$. Fixing suitable windows $g_j
\in L^2(\bR )$ and using vector-valued notation $\bof = (f_1, \dots ,
f_n)$, one then considers the sequence of numbers
 $\langle \bof , \pi_\lambda \bog \rangle =
  \sum _{j=1}^n \langle f_j, \pi_\lambda g_j\rangle $ 
  for $\lambda \in \Lambda $, i.e., the inner product in
$\vvlrd $. Roughly speaking, these numbers then measure the
time-frequency content of the whole  $\bof$ at the point $\lambda $ in the
time-frequency plane.

 Now one  requires that 
$\bof $ is completely determined by these inner products and that
there exists a stable reconstruction. This requirement leads to the
definition of a vector-valued Gabor frame~\eqref{frameinequalitya}.

The general problem of characterizing \emph{all} lattices $\Lambda $
for which $\cG (g, \Lambda )$ is a frame seems to be extremely
difficult. In fact, this problem is solved only for three classes of
basis functions, namely for the Gaussians  $H_0(t) = e^{-at^2},
a>0,$ in ~\cite{lyu,sw}, for the hyperbolic secant $g(t) =
(\cosh a t )\inv $ in~\cite{JS02}, and for  the one-sided exponential function $g(t) =
e^{-a|t|} \chi _{[0,\infty )}(t)$~\cite{janssen96}.  For the Gaussian,
 our understanding is  based on the  connection between the frame property
 of $\cG (g,\Lambda )$  and a  classical interpolation problem in the
 Bargmann-Fock  space of entire functions. The case of the hyperbolic
 secant is reduced to the case of the Gaussian. 

 For other choices of the $g$, where the connection to complex
 analysis is missing, the   conditions for $\cG(g,a\bZ\times b\bZ)$ to form a frame
 become very different and often rather intriguing~\cite{janssen02b}.

Even  less is known about Gabor superframes. Necessary density
conditions were studied in~\cite{balan2}, sufficient density
conditions can be derived from coorbit theory~\cite{fg89jfa} and from
the sampling theory on the Heisenberg group~\cite{fuehr07}. In particular, these
results  imply
that for any  sufficiently dense lattice $\Lambda $ and a mild condition
on the vector $\bog $ the Gabor system $\cG (\bog , \Lambda )$ is a
Gabor superframe.

In this article we study the frame property of $\cG(\boh , \Lambda)$
in  the case when $\boh $ is the vector of the first $n+1$   Hermite function
 $\boh =(H_0, \dots , H_n)$. In the scalar-valued case, the study of
 Gabor frames with the Gaussian window $H_0$ is natural, because the
 Gaussian minimizes the uncertainty principle. Likewise, for
 vector-valued frames, the study of Gabor superframes with the Hermite
 window $\boh $ is natural, because the first Hermite functions are
 the unique orthonormal set $\{f_0, \dots , f_n\}$ in $L^2(\bR )$ of size $n+1$ 
 satisfying the normalizations $\int_{-\infty}^\infty t |f_j(t)|^2 dt
 =0$ and $ \int_{-\infty}^\infty \xi |\hat{f_j}(\xi)|^2 d\xi =0$,
 $j=0, \dots , n$,   such that  the  uncertainty 
 \beq
 \label{sum_uncertainty}
 \sum _{j=0}^n \left ( \int_{-\infty}^\infty t^2 |f_j(t)|^2 dt +
 \int_{-\infty}^\infty \xi^2 |\hat{f_j}(\xi)|^2 d\xi  \right )
 \eeq
is minimized. 
Another motivation comes again from signal processing where Hermite
functions are used,  see~\cite{matz}.

The case of the Hermite vector window has been already investigated by
 F\"uhr by employing techniques related to sampling in the space of
bandlimited functions on the Heisenberg group.  He   proved the
following result.   

\noindent{\bf Theorem}\cite{fuehr07}.  {\em Let $\boh = (H_0 , \dots ,
  H_n)$.  There exists a constant $C>0$ 
  with the following property: If the diameter of the (smallest)
  fundamental domain of $\Lambda $ is less that $C(n+1)^{-1/2}$, then  
$\cG(\boh,\Lambda)$ is a frame for $L^2(\bR,\bC^{n+1})$.}

 Unfortunately nothing can be said about the constant $C$ within such
 an approach. F\"uhr uses the so-called ``oscillation method'' which
 is used for existence results, but  in general does not yield sharp
 results. 

In this paper, we give  a complete characterization of vector-valued
frames with Hermite functions. If the lattice is given as $\Lambda =
A\bZ ^2$ for some invertible, real-valued $2\times 2$-matrix $A$, let
$s(\Lambda )= |\det A |$ be the area  of the fundamental domain of
$\Lambda $. Our main result can be formulated as follows.

 \begin{thm}
 \label{mainchar1}
Let $\boh = (H_0 , \dots , H_n)$ be the vector of the first $n+1$
Hermite functions. Then   $\cG(\boh,\Lambda)$ is   a frame for
$L^2(\bR,\bC^{n+1})$, \fif\ 
 \beq
 \label{main_estimate}
 s(\Lambda)< \frac{1}{n+1}.
 \eeq 
\end{thm}

 By specializing to the $n$-th coordinate of $\boh$, we obtain a
  condition for scalar-valued Gabor frames with
Hermite functions. The following result was already announced in~\cite{glyu}.

\begin{prop}
\label{simple_hermite}
 If $s(\Lambda ) < \frac{1}{n+1}$, then $\cG (H_n , \Lambda )$ is a
 frame for $L^2(\bR )$. 
\end{prop}

  For $n=0$, the case of the
 Gaussian, we recover  (the lattice part of) the results
 in~\cite{lyu,sw}. Our method of proof yields several new results
 about the dual window. On the one hand, we construct a dual window
 for $\mathcal{G}(H_n, \Lambda )$ with Gaussian decay in time and
 frequency; on the other hand, we derive a new  estimate for the lower
 frame bound of $\mathcal{G}(H_0, \Lambda )$. 
Furthermore, we discuss  an example
 for $n=1$ which   suggests that the
 sufficient conditions of Proposition~\ref{simple_hermite} are sharp for all $n$.

In order to prove these results we combine the   techniques of
Gabor analysis and complex-analytic methods. The (now) classical 
structural results related to  the scalar Gabor frame systems can 
be formulated for the vector case as well.  They lead  to an interpolation
problem in the Fock space of entire functions.

 This problem is not  "purely holomorphic":   the values of
 linear combinations of  functions from the Fock space
  and their derivatives are prescribed in the lattice points; however,  the
 coefficients  of such combination  are antiholomorphic polynomials. 
The classical methods of complex analysis cannot be applied for problems
of this kind.  Fortunately,  in our particular case, i.e.,  for the
lattices of sufficiently 
small density,  one may use well-developed machinery of the elliptic functions.

The paper is organized as follows. In  Section~2 we collect the
necessary facts related to vector-valued frames, in particular, the 
construction of the  frame operator and    
the vector-valued  version of
the structural theorems:  Janssen's representation   and the  Wexler-Raz
biorthogonality criteria. The complex analytic tools  are presented in
Section~3, they include basics about Fock spaces and the growth
estimates of the  Weierstrass  $\sigma$-function. Section~4 contains
the proofs of Theorem  \ref{mainchar1} and Proposition  
\ref{simple_hermite}.  We also prove a result (Proposition \ref{counter}) indicating that
the density condition   in Proposition \ref{simple_hermite}  might be
sharp. The rest of Section~4 contains estimates of dual windows for
the Hermitian frames and of the lower frame bound.

\section{General Gabor Superframes}

  \subsection{Scalar case}

To derive the criterion for Gabor superframes from Hermite  functions
(Theorem~\ref{mainchar1}), we need some general notion and results 
related  to  Gabor superframes
such as  Janssen's  representation of the frame operator and  Wexler-Raz
identities.  Although these are among the fundamental  
results about Gabor frames, they have not yet been formulated  for
Gabor superframes. For convenience and  later reference we formulate
them explicitly for Gabor frames on $\rd $ (instead of $\bR $) in this 
section. Since the vector-valued case is   a simple consequence of the
scalar-valued case, we defer the (easy) proofs to the Appendix.  

\medskip
 
Let $\Lambda = A \bZ ^{2d}$,  be a lattice in $\bR ^{2d}$, where $A$
is a non-singular real  $2d \times 2d$-matrix. Let $s(\Lambda )
= |\det| A$ be  the volume of a fundamental domain of $\Lambda$.
The  {\em adjoint lattice} is defined by the commutant property  as 
\begin{equation}
  \label{eq:c2}
  \Lambda ^\circ = \{ \mu \in \bR ^{2d}: \pi (\lambda ) \pi (\mu ) =
  \pi (\mu ) \pi (\lambda ) \,\, \text{ for all } \, \lambda \in
  \Lambda \}\, .
\end{equation}
 If $\Lambda = \alpha \bZ ^d \times \beta \bZ ^d$, then $\Lambda
^\circ = \beta \inv \bZ ^d \times \alpha \inv \bZ ^d$. In general, for
$\Lambda = A \bZ ^{2d}\subset \bR^{2d}$, then
\begin{equation}
  \label{eq:cn1}
  \Lambda ^\circ =
\mathcal{J}(A^T)\inv \bZ ^{2d} \, ,
\end{equation}
 where $A^T $ is the transpose of $A$
and $\mathcal{J} =
\begin{pmatrix}
  0 & I \\ -I  & 0
\end{pmatrix}$ (consisting of $d\times d$ blocks) is the
matrix defining the standard symplectic
form~\cite{feichtinger-kozek98}. 

 Furthermore, from \eqref{eq:cn1} we see that  
\begin{equation}
  \label{eq:c7}
  s(\Lambda ^\circ ) = s(\Lambda ) \inv \, .
\end{equation}
The density of $\Lambda $ is defined as $d(\Lambda ) = s(\Lambda )\inv
$, so that $d(\Lambda )$ coincides with the usual notions of density.

\smallskip

In order to make the exposition self-contained  we recall the (now)
standard results about Gabor frames.  We refer the reader to Ch.7 in
\cite{Charlybook} for the complete proofs. 

  Given two functions (windows)  
$ g, \gamma  \in L^2(\rd)$,  the associated Gabor {\em frame
operator} is defined to be 
\beq
\label{foperator_scalar}
Sf=S^\Lambda_{g,\gamma}f=
      \sum_{\lambda \in \Lambda} \lng f,\pi_\lambda g \rng \pi_\lambda \gamma .
\eeq

For arbitrary $ g, \gamma  \in L^2(\rd)$ the right-hand side in \eqref{foperator_scalar}
defines a continuous operator from $\cS (\rd )$
to $\cS ' (\rd )$ with weak$^*$-convergence of the sum.
 
 Under  slight conditions on $g$ and $\gamma$, 
 for example if both $g$ and $\gamma$  belong to the Feichtinger algebra 
 $M^1$,    the series in \eqref{foperator_scalar} converges in the usual 
 $L^2(\rd)$-norm.  Recall that a function $g$ on $\rd $  belongs to the Feichtinger algebra 
 $M^1(\rd)$, if 
 \beq
 \label{m_one}
\|g \|_{M^1} :=  \int _{\bR^{2d} } |
\langle g, \pi _z \varphi  \rangle | \, dz < \infty \, , 
\eeq
where $\varphi (t) = 2^{d/4}\, e^{-\pi t^2}$ is the $L^2$-normalized
Gaussian. 
This condition is met if, for example,  both $g$ and its Fourier transform 
$\hat{g}$, decay sufficiently fast, 
 see  \cite{Charlybook}, Ch.~7 and also \cite{ grstudia} for
 discussion and the proofs.  

The convergence of \eqref{foperator_scalar} follows from the following
lemma taken from \cite{Charlybook} (the statement is not stated as
explicitly as we want it, but follows from combining
Propositions~11.1.4, 12.1.11, and 12.2.1). 

\begin{lemma} \label{bessel}
  If $g\in M^1(\rd )$, then, for each lattice $\Lambda$ 
  \begin{eqnarray}
    \label{eq:cn2}
 &    &     \sum _{\lambda \in \Lambda } |\langle f, \pi _\lambda g\rangle |^2
    \leq n(\Lambda ) \, \|g\|_{M^1} ^2 \, \|f\|_2^2 \\
\mathrm{and} \,\,\,\,\, &   &  \|\sum _{\lambda \in \Lambda } c_\lambda  \pi _\lambda
g\|_2 \leq n(\Lambda ) ^{1/2} \, \|g\|_{M_1} \, \|c\|_2 \, ,  \label{eq:cn3}
  \end{eqnarray}
where $n(\Lambda )  = \kappa \max _{k\in \bZ ^{2d}} \mathrm{card}\, (\Lambda \cap (k+[0,1]^{2d}))$ for some absolute constant $\kappa >0$. 
\end{lemma}

In other words if $g\in M^1(\bR^d)$, then  for each choice of $\Lambda$  the sequence $\mathcal{G}(g,\Lambda )$
is a {\em  Bessel sequence}, i.e.
\[
 \sum _{\lambda \in \Lambda } |\langle f, \pi _\lambda g\rangle |^2 \leq 
  \const \|f\|_2^2.
\]
 Moreover, the
upper frame bound of $\mathcal{G}(g,\Lambda )$ can be estimated
by $n(\Lambda )\, \|g\|_{M^1}^2$.


By a theorem of Rieffel \cite{rieffel88} and Janssen \cite{janssen95} the frame operator can be represented as a sum of time-frequency shifts over the adjoint lattice:
\beq
\label{janssen_scalar}
Sf=S^\Lambda_{g,\gamma}f=
    s(\Lambda)^{-1}  \sum_{\mu \in \Lambda^\circ} \lng \gamma,\pi_\mu g\rng
                    \pi_\mu f .
\eeq

For $g,\gamma \in M^1(\rd)$ the series in the right-hand side converges 
both  in  $L^2(\rd)$-norm and also in operator norm in   $L^2(\rd)$.

Given a frame  $\cG(g,\Lambda)$ we say that $\gamma \in L^2(\rd)$ is a
{\em  dual window }
if $S_{g,\gamma}^\Lambda = I$

Based on  representation \eqref{janssen_scalar}   one  obtains a criterion for   the system 
$\cG(g,\Lambda)$ to   form a frame in $L^2(\rd)$:

\begin{prop}[Wexler-Raz biorthogonality relation]
\label{WR_scalar}
Assume that both  $\cG(g,\Lambda)$ and $\cG(\gamma,\Lambda)$  are Bessel sequences in $L^2(\rd)$. Then $\cG(g,\Lambda)$  is a frame in $L^2(\rd)$ with dual window $\gamma$ if and only if
\beq
\label{biorthogonality_scalar}
\frac 1 {s(\Lambda)} \langle \gamma, \pi_\mu g\rangle = \delta_{0,\mu}, 
\    \text{for} \ \mu \in \Lambda^\circ.
\eeq
\end{prop} 

We remark that for $g,\gamma \in M^1(\rd)$ the sequences $\cG(g,\Lambda)$ and $\cG(\gamma,\Lambda)$ always are Bessel sequences for any choice of
 $\Lambda$ by Lemma~\ref{bessel}.

The Wexler-Raz relations yield an estimate for the lower frame bound
that deserves to be better known.

\begin{cor} \label{lowerbound}
  Assume that $\mathcal{G}(g,\Lambda )$ is a frame for $L^2(\bR ^d)$ with some dual
  $\gamma \in M^1(\bR ^d)$. Then the optimal lower frame bounded
  $A_{\mathrm{opt}} = \|S_{g,g}\inv \|_{L^2\to L^2}\inv$ satisfies
  \begin{eqnarray}
    \label{eq:cn5}
    A_{\mathrm{opt}} \geq n(\Lambda )\inv \,\|\gamma \|_{M^1} ^{-2} \, .
  \end{eqnarray}
\end{cor}

\begin{proof}
  Since $f= S^\Lambda_{g,\gamma }f = \sum _{\lambda \in \Lambda }
  \langle f, \pi _\lambda g\rangle \,  \pi _\lambda \gamma$,
  Lemma~\ref{bessel} implies that
$$
\|f\|_2^2 \leq n(\Lambda ) \|\gamma \|_{M^1}^2 \sum _{\lambda \in
  \Lambda } |\langle f, \pi _\lambda g\rangle |^2 \, ,
$$
whence the estimate for the lower frame bound follows. 
\end{proof}

 \subsection{Structure of Vector-Valued Gabor Frames }

For the case of  vector-valued Gabor system the frame-type operator
is defined as 
 \beq
\label{foperator}
S\bof=S^\Lambda_{\bog,\bga}\bof= \sum_{\lambda \in \Lambda}
     \lng \bof, \pi_\lambda \bog \rng \pi_\lambda \bga.
 \eeq
The convergence properties are the same as in the scalar case.  

\begin{prop}
\label{janssen_prop}
 {\rm (Janssen representation for Gabor superframes)} Let the 
windows $\bga=(\gamma_j)_1^n, \bog=(g_j)_1^n\in L^2(\bR ^d, \bC^n)$ be such that 
$\gamma_j, g_j\in M^1$, $j=1,2, \ldots \ ,n.$  Then  the frame type
operator $S_{\bog , \bga }$ associated to $(\bga, \bog; \Lambda)$ 
  can be written as
  \beq
 \label{janssen} 
S\bof = S^\Lambda_{\bga,\bog}\bof =
 \sum_{\mu \in \Lambda^\circ} \Gamma(\mu)\pi_\mu \bof,
\eeq
where $\Gamma(\mu)$ is the $n\times n $ matrix with entries
\beq
\label{gamma_mu}
\Gamma(\mu)_{kl}= s(\Lambda)^{-1}\lng \gamma_k, \pi_\mu  g_l \rng \, 
\eeq
and the sum converges in the operator norm on $\vvlrd $. 
 \end{prop}

Given a Gabor (super)frame  $\cG (\bog, \Lambda )$ with $\bog \in
L^2(\bR ^d, \bC^n)$,  we say that  $\bga \in  L^2(\bR ^d, \bC^n)$ is a 
{\em dual window}  (with respect to the lattice $\Lambda$) if 
\beq
\label{dual}
        S^\Lambda_{\bog, \bga }=   S^\Lambda_{\bga,\bog} = I.
\eeq

Dual windows always exist, a special choice is given by the
\emph{canonical dual window}  
$\bga ^\circ :=  \left (S_{\bog   , \bog }^\Lambda \right )  \inv \bog $;
as in the scalar case the invertibility of  
   $ S_{\bog    , \bog }^\Lambda  $  
   is an easy consequence of the frame property.
   
\smallskip

From Janssen's representation we obtain a criteria for the system
$\cG(\bog, \Lambda)  $ forms a frame in $L^2(\bR ^d,\bC^n)$ with dual window $\bga$,
that is a vector analog of the Wexler-Raz condition.

\begin{prop}
\label{wexler_raz}
{\rm (Wexler-Raz biorthogonality).}
Assume that both $\cG(\bog, \Lambda) $ and $   \cG(\bga, \Lambda) $
are Bessel sequences in $L^2(\bR ^d, \bC^n)$. Then  $\cG(\bog, \Lambda) $ is a vector-valued frame in 
$ L^2(\bR ^d, \bC^n)$ with dual window $\bga$ if and only if
\beq
\label{wr_orthogonality}
\frac 1 {s(\Lambda)} \lng \gamma_l, \pi_\mu g_j \rng =
       \delta_{0, \mu} \delta_{l,j} \quad \text{for} \quad
       \mu \in \Lambda^\circ, \ j, l = 1,2, \ldots , n \, . 
\eeq
\end{prop}

As in the scalar case we obtain Balan's  necessary density condition~\cite{balan2} by
adjusting an argument of Janssen~\cite{janssen94}. 

\begin{prop} \text{(Density theorem.)} 
\label{density}
 If $\cG(\bog, \Lambda)$ is a frame for 
$L^2(\bR ^d, \bC^n)$, then $s(\Lambda)\leq n^{-1}$.
\end{prop}
 

\begin{proof}
  Let $\bga ^\circ= S_{\bog , \bog } \inv \bog$   be
 the canonical dual window. Then $\bof = \sum_{\lambda \in \Lambda}
 \langle \bof, \pi (\lambda ) \bga ^\circ \rangle \pi_\lambda \bog
 $ holds for every $\bof \in \vvlrd $. If $\bof $ has also the representation  
$    \bof = \sum_{\lambda \in \Lambda} c_\lambda \pi_\lambda \bog$, 
  then  by \cite{DS52}
 \[
 \sum_{\lambda \in \Lambda} |\lng \bof, \pi_\lambda \bga \rng |^2 \leq
      \sum_{\lambda \in \Lambda} |c_\lambda|^2 \,.
    \]
    We apply this argument to the trivial expansion 
    $g= 1\cdot \bog+\sum_{\lambda \neq 0} 0\cdot \pi_\lambda \bog$. Thus we obtain
 \[
 |\lng \bog, \bga \rangle |^2\leq \sum_{\lambda \in
   \Lambda}|\lng\bog,\pi_\lambda\bga \rng |^2 \leq 1. 
 \]
The Wexler-Raz identities \eqref{wr_orthogonality} yield
\[
\lng \bog, \bga \rng = \sum_{j=1}^n \lng g_j, \gamma_j \rng = ns(\Lambda),
\]
and the result follows.
\end{proof}

The following duality result is not needed in the sequel, but is
included for completeness. It answers a question of our engineering
colleague G.~Matz, see e.g. \cite{matz}.

\begin{thm}[Janssen-Ron-Shen duality] \label{ron-shen}
  The Gabor system $\cG (\mathbf{g}, \Lambda )$ is a frame for 
  $L^2(\bR^d, \bC^n)$,
  \fif\ the union of Gabor systems $\bigcup _{j=1}^n \cG (g_j, \Lambda
  ^\circ )$  is a Riesz sequence for $\lrd $. 
\end{thm}

 \section{Complex Methods:  Fock space and the Weierstrass
   sigma-function}

The complex analytic techniques  we use   are concentrated around
 the Fock space   of  entire functions and precise estimates for
 the Weierstrass 
$\sigma$-function. These topics are closely related: roughly speaking   
$\sigma$-functions deliver examples of functions of "maximal possible growth" 
in  the Fock space.

\subsection{Fock space}
  
In this subsection we recall the basic properties of the Fock space
which are important for our applications. We  refer the reader to \cite{Folland}, 
\cite{Charlybook} for detailed proofs and also for discussion of numerous application
of this space in signal analysis and quantum mechanics.
 
 \begin{defin}
 \label{Fock}
 The Fock space $\cF$ is the Hilbert space of all entire functions such that
 \begin{equation}
 \label{Fock_definition}
 \|F\|^2_\cF=\int_ \bC |F(z)|^2 e^{-\pi |z|^2} dm_z <\infty,
 \end{equation}
 where $dm$ is the planar Lebesgue measure.
\end{defin}
The natural inner product in $\cF$ is  denoted by  $\lng \cdot,\cdot \rng_\cF$.

Below we list the properties of the Fock space, which will be used in the sequel. 

(a)  Each  point evaluation is a bounded linear functional in $\cF$, 
 and  the  corresponding reproducing kernel is the function $w\mapsto
 e^{\pi \bar{z} w }$, precisely, 
\beq
\label{reproducing} 
F(z) = \lng F(w), e^{\pi \bar{z} w } \rng_\cF \, ,  \qquad  \forall  F\in \cF.
\eeq

(b) The  functions from  $  \cF$ grow at most  like the  Gaussian
  (see e.g. \cite{LyuSe}),  more precisely,  
\beq
\label{growth_estimate}
|F(z)| = o(1) e^{\frac \pi 2 |z|^2}, \ \text{as} \ z \to \infty, \ F\in \cF.
\eeq

(c) The collection  of monomials 
\beq
\label{monomials}
e_n(z)=\Big(\frac{\pi^n}{n!}\Big)^{1/2} \, z^n, \ n=0,1,\ldots
\eeq 
forms an orthonormal basis in $\cF$. 

(d) Define the \BT\ of a function $f\in L^2(\bR)$  by 
\beq
 \label{bt}
  f \mapsto \cB f(z)=F(z)= 2^{1/4}e^{-\pi z^2 /2}\int_\bR f(t)e^{-\pi
    t^2}e^{2\pi t z} dt. 
  \eeq
%


\begin{prop}
\label{unitary}
The \BT is a unitary mapping between $L^2(\bR)$ and $\cF$. 
\end{prop}

(e) The  \HFS\  are defined by  
\beq
\label{HFS}
H_n(t)= c_ne^{\pi t^2} \frac {d^n}{dt^n} e^{-2\pi t^2}, \ n=0,1,2, \dots
, .
\eeq   
where the coefficients $c_n$ are chosen in order to have $\|H_n\|_2=1$. 
It  is a classical result that  the set of Hermite functions
$\{H_n\}_{n=0}^\infty$ forms an  orthonormal basis in  $L^2(\bR)$. 
 

Their image under the \BT is $\{e_n\}_{n=0}^\infty$  -- the natural orthonormal basis in $\cF$:
\beq
\label{image}
\cB H_n (z) = e_n(z), \ n=0, 1, \ldots \ .
\eeq

(f) In what follows we  identify $\bC$ and $\bR^2$. In particular
for each $\zeta=\xi +i\eta \in \bC$ we write 
$\pi_\zeta=\pi_{(\xi,\eta)}$. 
 Define the shift $\beta_\zeta: \cF \to \cF$ in the \FS\  by
\beq
\label{fock_shift}
 \beta_\zeta F(z) = e^{i\pi \xi \eta} e^{-\pi |\zeta|^2/2}e^{\pi {\zeta} z} F(z-\bar{\zeta}).
 \eeq
Then $\beta _\zeta $ is unitary on $\cF $,  and the   \BT intertwines
the \FS shift and the \TFS: 
\beq
\label{intertwines}
\beta_\zeta \cB = \cB \pi_\zeta.
\eeq

Based on this proposition one can easily obtain inner product of a
function $f\in L^2(\bR)$ with the time-frequency   shifts of a  
Hermite function as in~\cite{BS93}. This quantity is    the
short-time Fourier 
transform with respect to a  Hermite function.

\begin{prop}  
\label{hermite_stft}
Let $f\in L^2(\bR)$ and  $F(z)=\cB f(z)$. Then, for all $\zeta \in \bC$,
\beq
\label{brcmn_stft}
\lng f, \pi_\zeta H_n \rng_{L^2(\bR)}=  
               \frac 1{\sqrt{\pi^n n!}}e^{-i\pi \xi \eta} e^{-\pi |\zeta|^2/2}
                    \sum_{k=0}^n \binom{n}{k}  (-\pi {\zeta})^k F^{(n-k)}(\bar{\zeta}) .
\eeq
\end{prop}
\begin{proof}
  Using the intertwining property~\eqref{intertwines}, we have 
\begin{eqnarray*}
  \lng f, \pi_\zeta H_n \rng_{L^2(\bR)}&=&  \lng F, \beta_\zeta \cB
  H_n \rng_\cF \\
&=&     \Big(\frac {\pi^n}{n!} \Big)^{1/2} \, e^{- i\pi \xi \eta}\,  e^{-\pi |\zeta|^2/2} 
      \lng F(z), e^{\pi {\zeta} z} (z-\bar{\zeta})^n\rangle \\
&=&  \Big  (\frac {\pi^n}{n!} \Big  )^{1/2} e^{-i\pi \xi \eta}
e^{-\pi |\zeta|^2/2}\sum_{k=0}^n \binom{n}{k}   (- {\zeta})^k
\lng F(z), z^{n-k}e^{\pi {\zeta} z}\rangle      \\
&=&   \Big (\frac {\pi^n}{n!} \Big )^{1/2} e^{-i\pi \xi \eta} e^{-\pi |\zeta|^2/2}\sum_{k=0}^n \binom{n}{k} (-  {\zeta})^k  \pi^{-n+k} \frac{d^{n-k}}{d\bar{\zeta}^{n-k}}  \lng F(z),e^{\pi {\zeta} z}\rng.       
\end{eqnarray*}
  It remains to apply relation  \eqref{reproducing}.
  \end{proof}
\smallskip
 
Finally we give in this section a description of the space  $M^1$ in
the Bargmann transform  terms.  Since $|\langle f , \pi _{\bar{z}} H_0
\rangle | = |\mathcal{B}f(z)| \, e^{-\pi |z|^2/2}$
(e.g.,\cite{Charlybook} or \cite{Folland}), we obtain
the following.

\begin{prop}
\label{p_one_barg}
A function $f\in L^2(\bR)$  belongs to $M^1$ if and only if
its Bargmann transform $F(z)=\cB f$ satisfies
\beq
\label{m_one_barg}
\|f\|_{M^1} = \int_\bC |F(z)|e^{-\pi |z|^2 / 2} dm_z < \infty.
\eeq
\end{prop}  
This is just a reformulation of condition \eqref{m_one}.

\subsection{Estimates of $\sigma$-function}
\label{estimates}

 In this section we collect  definitions and known facts
 about  the Weierstrass functions.  They will be used in the next section.

 
 Given two numbers   $\omega_1, \omega_2 \in \bC $ such that 
  \beq
 \label{positive}
 \Im  \Big (  \omega_2/\omega_1 \Big ) >0.
 \eeq
 we consider the  lattice $\Lambda=\{m_1\omega_1+m_2\omega_2; m_1,m_2 \in
 \bZ\}\subset \bC$. The numbers    $\omega_1, \omega_2 \in \bC $ are
 called {\em periods }  of the lattice $\Lambda$. 
 
   In the previous section we used   lattices of  the form
 $\Lambda = A \bZ^2 \subset \bR^2$, where $A$ is an invertible 2$\times$2 matrix.
   The  two constructions coincide after the natural identification of  
$\bR^2$ and $\bC $,  if  $\det A >0$. In this case    $\omega_1, \omega_2$
correspond to the columns of  $A$.  

Let $Q=\{(x,y); 0\leq x <1, 0\leq
y <1\}$. 
The parallelogram $\Pi _\Lambda = AQ$ with based on 
$\omega_1$ and $\omega_2$ is a fundamental domain  for  $\Lambda$
(also called the {\em period    parallelogram}).  
 Its area can be  expressed through the periods as 
\beq \label{area}
s(\Lambda ) = \text{Area}(\Pi _\Lambda)= \det A=  \Im (\bar {\omega}_1\omega_2)
= -\frac {i} 2 (\bar{\omega}_1\omega_2- \bar{\omega}_2\omega_1).
  \eeq

Next  consider the following Weierstrass functions:
 \beq
 \label{Pfunction}
\cP(z)= \frac 1{z^2} + \sum_{\omega \in \Lambda'}   \left \{
      \frac 1 {(z-\omega)^2} - \frac 1 {\omega^2}  \right \},
\eeq
\beq 
\label{zeta}
\zeta(z)= \frac 1 z + \sum_{\omega \in \Lambda'} 
     \left \{ \frac 1 {z-\omega} + \frac 1 \omega + \frac z {\omega^2 } \right \},
\eeq
and
\beq
\label{sigma}
\sigma(z)= z \prod_{\omega \in \Lambda'}  \left ( 1 - \frac z \omega \right )
    e^{\frac z \omega + \frac {z^2}{2\omega^2}},
\eeq
where  $\Lambda'= \Lambda \setminus \{0\}$. The lattice  
$\Lambda$ plays a special role for the Weierstrass functions: $\cP$  has poles
of order two precisely on $\Lambda $,  $\zeta$ has  poles
of order one, and  $\sigma$ has simple zeros precisely on $\Lambda $.

We list some  basic properties of the Weierstrass  functions. See,
e.g.,  \cite{Ahiezer} for detailed proofs.

\begin{itemize}
\item 
The $\cP$-function  is  elliptic, i.e., 
\beq
\label{P1}
\cP(z)=\cP(z+\omega_k), \quad z\in \bC, \ k=1,2.
\eeq
\item 
The function $\zeta$ changes by a constant when its argument changes by a lattice point:
\beq
\label{eta}
\eta_k:=\zeta(z+\omega_k)- \zeta(z)= \const, \ k=1,2.
\eeq
\item The constants $\eta_k$ 
satisfy the  Legendre relation:
\beq
\label{legendre}
\eta_1 \omega_2 - \eta_2 \omega_1 = 2\pi i .
\eeq
\item
The Weierstrass $\sigma$-function  is quasi-periodic, it satisfies the
following relation  
\beq
\label{sirel}
\sigma(z+\omega_k)= - \sigma(z) e^{\eta_kz + \frac 12 \eta_k \omega_k}
\qquad k=1,2 \, .
\eeq
\end{itemize}

\medskip
In order to understand  the growth properties  of the Weierstrass
$\sigma $-function in dependence of the lattice 
$ \Lambda$, we follow an elegant argument of Hayman~\cite{Hayman}. He   realized
that, after a proper normalization and growth
compensation, the  absolute value of $\sigma_\Lambda $ becomes a doubly
periodic function.  We include   here  this arguments  in order to make our presentation self-contained.
 \begin{prop}
 \label{lemmasigma}
 Set
 \beq
 \label{sconstants}
 \alpha(\Lambda)=  \frac{i\pi}{\bar{\omega}_1\omega_2- \bar{\omega}_2\omega_1}=
 \frac{\pi}{2 s(\Lambda)},
 \quad
 a(\Lambda)= \frac 12 \frac {\eta_2\bar{\omega}_1-\eta_1 \bar{\omega}_2}
                      {\omega_1\bar{\omega}_2-\omega_2 \bar {\omega}_1},
\eeq
and 
\beq
\label{sigmaa}
\sigma_\Lambda(z)= \sigma(z)e^{a(\Lambda)z^2}.
\eeq
Then the function $|\sigma_\Lambda(z)|e^{-\alpha(\Lambda) |z|^2}$ is  periodic  with periods
$\omega_1$ and $\omega_2$.
\end{prop}
 
 \medskip
The periodicity implies a growth estimate for the
modified sigma-function $\sigma _\Lambda$.

\begin{prop}
\label{growth}
  Set 
  \beq
  \label{C}
  c(\Lambda )= \sup _{z\in \Pi _\Lambda }
  |\sigma_\Lambda(z)| e^{-\alpha(\Lambda) |z|^2}. 
  \eeq
   Then 
\beq
\label{estimate}
|\sigma_\Lambda(z)| \leq c(\Lambda ) \, e^{\frac{ \pi}{ 2 s(\Lambda)}|z|^2},
\quad \forall z\in \bC\, .
\eeq
For  each  $\epsilon >0$ we have 
\footnote{Here and in what follows the sign $\asymp$ means that the ratio of the left- and right-
hand sides is located between two positive constants.}
\beq
\label{asymp}
|\sigma_\Lambda(z)|\asymp e^{\alpha(\Lambda ) |z|^2}, \quad \text{whenever dist}(z,\Lambda)>\epsilon.
\eeq
\end{prop}
   
\begin{proof}  The function 
$|\sigma_\Lambda(z)|e^{-\alpha(\Lambda) |z|^2}$ is bounded in 
$\Pi_\Lambda$ and has its only zeros at the vertices of $\Pi _\Lambda
$. Hence it is bounded away from $0$ on every compact subset of $\Pi
_\Lambda$,  which does not contain its vertices 
and relation \eqref{asymp}  follows now by the
periodicity. 
\end{proof}
\medskip

\noindent {\em Proof of Proposition \ref{lemmasigma}.} It follows from \eqref{sirel} that
\[
\Big |  \sigma(z+\omega_k) e^{a(\Lambda)(z+\omega_k)^2}   \Big |
e^{-\alpha(\Lambda)|z+\omega_k|^2}=
\Big |  \sigma(z) e^{a(\Lambda)z^2}\Big |e^{-\alpha(\Lambda) |z|^2} e^{\Re A_k(z)}, \ k=1,2,
\]
where 
 \beq
\label{ak}
A_k(z)=  z\Big (\eta_k+2a(\Lambda)\omega_k -  
          2\alpha(\Lambda) \bar{\omega}_k\Big )  +
                 \Big ( \frac 12 \eta_k \omega_k  +a(\Lambda)\omega_k^2-
                                    \alpha(\Lambda) |\omega_k|^2 \Big ),   \ k=1,2.
\eeq
An explicit calculation, based on relations \eqref{sconstants}  and \eqref{legendre}
shows that 
\[
\eta_k+2a(\Lambda)\omega_k -  2\alpha(\Lambda) \bar{\omega}_k= 0, \ k=1,2,
\]
and also
\[
\Re \Big (\frac 12 \eta_k \omega_k  +a(\Lambda)\omega_k^2-
        \alpha(\Lambda) |\omega_k|^2\Big )= 0, \
k=1,2.
\] 
Proposition \ref{lemmasigma} now follows.
\hfill $\Box$

\bigskip


\section{Frames with Hermite functions.}

\subsection{Hermite superframes}

In this section we prove  Theorem  \ref{mainchar1}   and thus give a complete
characterization of Gabor superframes with Hermite functions.

We divide the proof into several steps. 
First  we will use the matrix form of the Wexler-Raz  biorthogonality relations
  (\ref{wr_orthogonality}) and  translate these relations  to
  an interpolation problem on  Fock space.

Let, as earlier,  $\boh=(H_j)_{j=0}^n$ be the vector-valued window consisting  of the
first $n+1$ Hermite functions.  
  For each  
$\bga=(\gamma_j)_{j=0}^n\in L^2(\bR,\bC ^{n+1})$ denote
\label{dual_gamma}
\beq G_j=\cB\gamma_j,  \ j=0,1,\ldots \ ,n.
\eeq
Taking $\bog=\boh$ in the biorthogonality relations
\eqref{wr_orthogonality} and using 
Proposition \ref{hermite_stft},   we can rewrite these relations as an
interpolation problem  for functions $G_j\in \cF$:

\begin{eqnarray}
\label{new_relations}
\lng \gamma_j, \pi_\mu H_l\rng_{L^2(\bR)} &=&  e^{-i\pi \Im \mu \Re \mu}e^{-\pi|\mu|^2/2} 
       \frac 1 {\sqrt{\pi^l l!}}\sum_{k=0}^l \binom lk
       (-\pi{\mu})^kG_j^{(l-k)}(\bar{\mu})\\ 
&=&     s(\Lambda)\delta_{\mu,0}\delta_{j,l}, \quad j,l=0,1,\ldots \
              ,n;  \ \mu\in \Lambda^\circ.    \nonumber
\end{eqnarray}

Thus $\cG(\boh, \Lambda)$ is a frame for $L^2(\bR, \bC ^{n+1})$, if
and only if there exist functions   $G_j\in \cF$ satisfying \eqref
{new_relations} and if $\cG (\bga ,\Lambda )$ is a Bessel sequence in
$\L^2(\bR, \bC^{n+1}) $. 

This interpolation problem can be rewritten in a  simpler way. Indeed, 
for each  
$\mu\in \Lambda^\circ\setminus \{0\}$  and    $j\in
\{0,1,\ldots \ ,n\}$, 
the system \eqref {new_relations}   is a triangular linear system in
the variables 
$\{G_j^{(m)}(\bar{\mu})\}_{m=0}^n$ with non-zero diagonal coefficients and
zero right-hand side. Clearly, it has just zero solutions. Thus if
$\boG=(G_j)_{j=0}^n$ satisfies  
\eqref {new_relations}, then $G_j^{(m)}(\bar{\mu})=0, \ $
$j,m=0,1,\ldots  ,n, \ $ $\mu\in \Lambda^\circ\setminus \{0\}$. For
$\mu=0$, the   relations  \eqref {new_relations} take the form 
\[
\frac 1 {\sqrt{\pi^l l !}} G_j^{(l)}(0)= s(\Lambda) \delta_{j,l}, \ j,l=0,1, \ldots \ ,n.
\] 
By adjusting the normalization of the $G_j$,  we obtain the following
statement.

\begin{prop}
\label{delta_prop}
     The Gabor system  $\cG (\mathbf{h} , \Lambda ) $ is a  Gabor superframe for
$L^2(\bR , \bC ^{n+1})$ \fif\ there exist $n+1$ 
 functions $\gamma _j \in \lrd $  such that $\cG(\gamma _j, \Lambda )$ is a
 Bessel sequence and   $G_j=B\gamma _j, j=0, \dots , n$,  satisfy the
 (Hermite) interpolation problem
\begin{equation}
  \label{eq:herm}
G_j^{(\ell )} (\bar{\mu}) = \delta _{\mu ,0} \, \delta _{j,\ell} \qquad
\text{ for } \, \mu \in \Lambda ^{\circ}, j,\ell = 0, \dots, n \, .
\end{equation}
 \end{prop} 

\smallskip

\begin{proof}[Proof of Theorem~\ref{mainchar1}]
  \noindent{\em Sufficiency } of \eqref{main_estimate}.
We use the Weierstrass functions described in  Section \ref{estimates}.
 Construct the Weierstrass functions corresponding the lattice
 $\overline{\Lambda^\circ}$, as  given by \eqref{eq:c2}.  Let the constants
 $\alpha(\overline{\Lambda^\circ}), a(\overline{\Lambda^{\circ}})$ 
 be defined by  \eqref{sconstants}  and the
 function $\sigma_{\overline{\lambda^\circ}}$ be defined by \eqref{sigmaa}, 
 where the $\sigma$-function is constructed for the lattice
 $\overline{\Lambda^\circ}$. Consider the function 
 \beq
 \label{Sfunction}
 S(z)=\big ( \sigma_{\ola}(z)  \big )^{n+1}=
         \sigma   (z)^{n+1}
e^{a(\ola)(n+1)z^2}.
 \eeq
 
 The zero set of $S$ is  $\ola$,  and  the
 multiplicity of each zero is precisely  
 $n+1$. 
In our case $\alpha(\overline{\Lambda^{\circ}})
 =  \frac{\pi}{2 s(\Lambda ^\circ )} =  \frac \pi 2
s(\Lambda)$ and the growth estimates \eqref{estimate} and
\eqref{asymp} can be rewritten as  
 \begin{eqnarray}
\label{Sestimate}
|S(z)| &\leq & e^{\frac \pi 2 (n+1) s(\Lambda)|z|^2}, \  \text{for
  all} \ z\in \bC, \\
   |S(z)| &\asymp &  e^{\frac \pi 2 (n+1) s(\Lambda)|z|^2},
  \quad \text{ if  }   \dist(z,\ola)>\epsilon. 
  \label{Sestimate2}
  \end{eqnarray}

Now assume that $s(\Lambda ) < (n+1)\inv $. Then the functions 
 \beq
 \label{Sj}
 S_m(z) = \frac 1 {z^{n+1-m}} S(z), \ m=0,1,\ldots \ ,n
 \eeq
 belong to $\cF$. They have zero of order $n+1$ at each $\mu \in
 \ola \setminus \{0\}$,  and also 
 \[
 S^{(l)}_m(0)=0, \  \text{for} \ 0\leq l < m-1, \ \text{and} \ S^{(m)}_m(0)\neq 0.
 \]
 
 The solutions $G_j$ to the interpolation problem \eqref{eq:herm} can
 be now found in the form 
 \beq
 \label{solutions}
 G_j=\sum_{m=j}^n c_{m,j}S_m, \,\,\,\,  j=0,1,\ldots \ ,n.
 \eeq
 
  These functions have a zero of multiplicity $n+1$ at all points from 
  $\ola \setminus \{0\}$, while, for each $j$,  the  condition 
  $G_j^{(l)}(0)=\delta_{j,l}$ leads one to a triangular system of linear equations 
  with respect to the coefficients $ \{c_{m,j}\}_{m=j}^n$ with non-zero diagonal entries. 
  Clearly this system has a (unique) solution.
  
It remains to mention that each  system
$\cG(\gamma _j, \Lambda )$ is a  Bessel sequence.  Each $G_j$
inherits its growth from the functions $S_j$ and from $S$, thus 
\eqref{Sestimate} and Proposition~\ref{p_one_barg} imply that $\gamma
_j \in M^1(\bR ) $ for $j=0, \dots , n$. Now apply  Proposition~\ref{bessel}. 
   
  \smallskip
  
  \noindent{\em Necessity.}
 Let now $s(\Lambda)\geq (n+1)^{-1}$. The growth estimate~ 
  \eqref{Sestimate2} implies that 
  \beq
  \label{belowestimate}
   |S(z)| \geq \const  e^{\frac \pi 2  |z|^2},
 \quad \text{whenever} \,  \dist(z,\ola)>\epsilon.
\eeq
 
Now assume that  $\cG(\boh, \Lambda)$ is a frame for $L^2(\bR,\bC
^{n+1})$. Then  there exists a system of functions
$G_0, G_1,\ldots \ ,G_n\in \cF$ which satisfy the interpolation
problem~\eqref{eq:herm}. The function 
$G_n$ has zeros of multiplicity $n+1$ at $\Lambda^\circ\setminus \{0\}$ and
a zero of multiplicity $n$ at the origin. Therefore 
\[
\Phi(z)= \frac {zG_n(z)}{S(z)}
\]
is an entire function. Estimates \eqref{belowestimate}
and~\eqref{growth_estimate}  yield 
\[
|\Phi(z)| = o(|z|), \ z \to \infty, \ \dist(z, \ola) >\epsilon.
\]
By the maximum principle, the  restriction  
$ \dist(z, \Lambda^\circ) >\epsilon$   can be removed,
and $\Phi $ is in fact a bounded entire function. The
Liouville theorem implies that  $\Phi$ is a constant, or 
\[
G_n(z)= C \frac{S(z)}z.
\]
But  $z^{-1}S(z) \not \in \cF $, this follows again from \eqref {belowestimate}. Therefore
$C=0$ and hence $G_n$ is identically zero in contradiction to  \eqref{eq:herm}. This completes the proof of necessity.
\end{proof}

\noindent{\em Remark.} It follows from the density theorem
(Proposition \ref{density} ) that $\cG (\boh , \Lambda )$ cannot be a frame for
$s(\Lambda ) > \frac{1}{n+1}$. Our proof shows that we must also
exclude the case of the critical density $s(\Lambda ) =
\frac{1}{n+1}$. This is  a Balian-Low type phenomenon (see the
review~\cite{BHW95}). We conjecture that  in general, if $g_j \in M^1,
j=1, \dots , n$ and $\cG (\bog , \Lambda )$ is a Gabor frame, then
$s(\Lambda ) < \frac{1}{n}$,   
but we will not pursue this question
further in this work.

\subsection{Hermite frames}
Next  we consider scalar-valued Gabor frames with Hermite
functions.  
 
We remark that if for some $\bog \in L^2(\bR, \bC^{n+1})$ and lattice 
$\Lambda \subset \bR^2$ the system 
 $\cG(\bog, \Lambda)$ is a vector-valued frame in 
$L^2(\bR, \bC^{n+1})$, then, trivially,  for each $j=0,1,\ldots \ ,n$, the Gabor system
$\cG(g_j, \Lambda)=\{\pi_\lambda g_j, \lambda \in \Lambda\}$ is a frame for 
$L^2(\bR)$. More generally, for each 
 $\mathbf{c}\in \bC ^{n+1}$,  $ \mathbf{c}\neq 0$, the system
 $\cG(h, \Lambda)=\{\pi_\lambda h, \lambda \in \Lambda\}$  with 
$h = \sum _{j=0}^n c_j g_j$  also is a frame in  $L^2(\bR)$.
By applying this observation to the window $\bog = (H_0,\ldots \ , H_n)$
we obtain Proposition \ref{simple_hermite}, which gives
a condition for the one dimensional system with the Hermite window to be a frame for  $L^2(\bR)$. Actually a    slightly more general result   related to   linear
combinations of Hermite functions holds true

\begin{prop}
  \label{hermitegeneral}
 Let $n\in \bZ , n\geq 0$ and $h = \sum _{k=0}^n c_k H_k$, \, $\sum_0^n |c_k|\neq 0$.   If
 $s(\Lambda ) < (n+1)^{-1}$,    then
$\cG (h , \Lambda ) $ is a frame for $L^2(\bR )$.
\end{prop}
 
 \noindent{\em Remark.} Clearly the estimate above is not optimal in case 
 $c_n=0$.
 \smallskip

Amazingly enough the sufficient density of Theorem~\ref{hermitegeneral} might be
sharp, as is suggested by the following counter-example.

\begin{prop}
  \label{counter}
If $\Lambda = a \bZ \times b \bZ $ and $ab = 1/2$, then 
$\cG (H_{2n+1}, \Lambda )$ is not a Gabor frame for $L^2(\bR)$ for all
integers $n\geq 0$. 
\end{prop}
\begin{proof} 
We use a Zak transform argument. For $a>0$ the Zak transform is defined as 
\beq
\label{zak}
Z_a f(x,\xi)=\sum_{k=-\infty}^\infty f(x-ak)e^{2i\pi a k \xi}.
\eeq
We refer the reader to  \cite[Sec.~8.3]{Charlybook} for a detailed
discussion of  its properties.   In particular it is a unitary mapping
between  $L^2(\bR)$ 
 and $L^2(Q_{a,1/a})$ -- space of square summable functions in the  rectangle
 $Q_{a,1/a}=[0,a)\times[0,1/a)$.
For the case $ab=1/2$ (more generally for the case $(ab)^{-1}\in \bZ$)
the Gabor   frame operator is unitarily equivalent to a
multiplication operator  on $
L^2(Q_{a,1})$ by means of  the Zak transform. Precisely, 
\[
Z_a S_{H_n,H_n} f(x,\xi) =
     \big ( |Z_aH_n(x, \xi)|^2 + |Z_a H_n (x-\frac a2, \xi)|^2 \big ) 
               Z_af(x,\xi), \quad f \in   L^2(\bR).
\]                  
Consequently  $\cG (H_n, \Lambda )$ is a frame for  $L^2(\bR)$ \fif\
\begin{multline*}
 \quad 0< \inf_{Q_{a,1}} \{ ( |Z_aH_n(x, \xi)|^2 + |Z_a H_n (x-\frac
 a2, \xi)|^2 \big )  \}   \\ 
\leq \sup_{Q_{a,1}} \{ ( |Z_aH_n(x, \xi)|^2 + |Z_a H_n (x-\frac a2, \xi)|^2 \big )  \} <\infty,
         \quad 
 \end{multline*}
 see relation (8.21) in~\cite{Charlybook}. 

Clearly,   $Z_aH_n(x,\xi)$ is 
 a continuous function.  Since $H_{2n+1}$ is an odd function, its Zak
 transform   satisfies $Z_aH_{n+1}(0,0)=Z_aH_{n+1}(\frac a2, 0)=0$. This
 contradicts  the above criterium, therefore $\cG (H_{2n+1}, \Lambda
 )$ cannot be a frame.  
 \end{proof}

In~\cite[Prop.~3.3]{glyu} we showed that $\mathcal{G}(H_n, \Lambda )$ is a
frame, if and only if there exists a $\gamma_n \in L^2(\bR )$, such
that $\mathcal{G}(\gamma _n, \Lambda )$ is a Bessel system and the
Bargmann transform $G_n = \mathcal{B}\gamma _n$ is in $\mathcal{F}$ and  satisfies the interpolation
problem
\begin{equation}
  \label{Eq:3.2}
  \sum _{k=0}^n   \binom{n}{k} \,
         (-\pi \mu )^k G_n^{(n-k)} (\bar{\mu })   = \delta _{\mu , 0} \, \qquad
         \,\, \forall \mu \in \Lambda ^\circ \, .
\end{equation}
 
Note that \eqref{Eq:3.2} coincides with   \eqref{new_relations} for $j=l=n$
and that for $s(\Lambda ) < 1/(n+1)$ the function
\begin{equation}
  \label{eq:dualcc}
  G_n(z) = S_n(z) = c z^{-1}\sigma_{\overline{\Lambda^\circ}}^{n+1}(z)
\end{equation}
is a solution of \eqref{Eq:3.2} in $\mathcal{F}$. 

If $s(\Lambda ) \geq 1/(n+1)$, then we do not know whether
\eqref{Eq:3.2} has any solution in $\mathcal{F}$. The difficulty is
that this interpolation problem is not entirely holomorphic, and
the standard complex variable methods do not seem sufficient to investigate this
problem. In the light of Proposition~\ref{counter} it is conceivable
that the sufficient condition  $s(\Lambda) < 1/(n+1) $ in
Proposition~\ref{hermitegeneral} is also necessary.

\subsection{Norms and Decay of the Dual Window}

In the rest of this section we assume that $ s(\Lambda)<\frac 1 {n+1}$,
and  we estimate  the $L^2$ and $M^1$-norms of the
dual window for the frame  
$\cG (H_n , \Lambda )$.  Though the estimates are  simple, they  seem
to be new even for the Gaussian case  $n=0$.

 \begin{lemma} \label{norms}
Assume that $ (n+1)s(\Lambda)<1$ and let $\gamma _n$ be the dual
window with Bargmann transform $G_n$ defined in~\eqref{eq:dualcc}. 
Setting  $\kappa = \frac{1-(n+1)s(\Lambda)}{3-(n+1)s(\Lambda)}$, then 
\begin{equation}
  \label{eq:cne2}
    |\gamma_n(t)|+ |\widehat{\gamma _n}(t)| \leq C e^{-\pi \kappa t^2}
    \, . 
\end{equation}
Furthermore, 
\beq
\label{m_one_norm}
\|\gamma\|_{M_1} \leq  4  c({\overline{\Lambda ^\circ}})^{n+1}  \frac{
  1}{1-(n+1)s(\Lambda   )} , 
\eeq  
and
 \beq
\label{focknorm} 
\|\gamma _n\|_{L^2(\bR)}^2 \asymp    \log
\frac{1}{1-(n+1)s(\Lambda )}, 
\eeq
where $c(\Lambda )  $ is defined in~\eqref{C}.
\end{lemma}
\begin{proof} 
The estimates involve a routine calculation with Gaussians and are a
consequence of the growth estimate~\eqref{estimate} of Proposition~\ref{growth}.
Set $\rho = 1-(n+1)s(\Lambda )$. 

For the decay property, we use the inversion formula for the
short-time Fourier transform (e.g.,~\cite[Prop.~3.2.3]{Charlybook})
which states
$$
\gamma _n(t) = \int _{\bC } \langle \gamma _n , \pi _z H_0\rangle \,
\pi _z H_0(t) \, dm_z
$$
with absolute convergence of the integral for each $t\in\bR $. Since
by \eqref{estimate} 
$$
|\langle \gamma _n , \pi _{\overline{z}} H_0\rangle| =
|\mathcal{B}\gamma _n(z) |  e^{-\pi |z|^2/2} \leq C e^{-\pi \rho
  |z|^2/2}
$$
and since $\pi _{x+i\xi} H_0(t) = e^{2\pi i \xi t} H_0(t-x)$, we obtain
that
$$
|\gamma _n (t) | \leq C \int _{\bR } e^{-\pi \rho \xi ^2/2}\, d\xi \int _{\bR
} e^{-\pi \rho x ^2/2} e^{-\pi (t-x)^2} \, dx \, .
$$ 
The convolution of the two Gaussians in this integral is a multiple of
the Gaussian $e^{-\pi \frac{\rho}{2+\rho}t^2}$ by the semigroup
property of Gaussians with respect to
convolution~\cite[Lemma~4.4.5]{Charlybook}.

The $M^1$-norm of  $\gamma _n$ is readily estimated by 
\begin{eqnarray*}
  \|\gamma _n\|_{M^1} &=& \int _{\bC } |G_n(z) | \, e^{-\pi 
    |z|^2/2}\, dm_z \\
&\leq & c(\overline{\Lambda ^\circ})^{n+1}  \int _{\bC } \frac{1}{|z|}
\, e^{-\pi \rho     |z|^2/2}\, dm_z  \\
&=& 2\pi c(\overline{\Lambda ^\circ})^{n+1} \int _0^\infty  e^{-\pi
  \rho r/2}\, dr = 4 c(\overline{\Lambda ^\circ})^{n+1}
\frac{1}{\rho}\, .
\end{eqnarray*}

For the estimate of $\|\gamma _n \|_2^2$ we have similarly
\[
\|\gamma\|_{L^2(\bR)}^2 = \|G\|^2_\cF =
\int_{\bC} |z|^{-2} |\sigma_{\overline{\Lambda^\circ}} (z)|^{2(n+1)} e^{-\pi |z|^2} dm_z.
\]
When $|z|\leq 1$  the integrand is bounded uniformly, 
 so this 
part does not bring essential contribution into the whole norm. It
follows now from  
\eqref{asymp} that
\begin{multline*}
\|\gamma\|_{L^2(\bR)}^2 \asymp 
\int_{|z|>1}|z|^{-2} e^{-\pi \rho  |z|^2 } dm_z  \asymp
 \int_1^\infty  \frac 1 r e^{-\pi \rho r^2} dr 
    \asymp  - \log (1-(n+1)s(\Lambda ))\, . 
 \end{multline*}
This is the desired inequality.
\end{proof}
\noindent{\em Remarks.} 1. In general, the dual window $\gamma _n$ is not
the canonical dual window (given by $S_{H_n,H_n}\inv
H_n$). Janssen~\cite{janssen96} observed that for $n=0$ and $\Lambda = a\bZ
\times (aN)\inv \bZ $ with a positive integer $N$ the canonical dual
window decays at most exponentially, whereas the dual window
constructed by means of the $\sigma $-function possesses Gaussian
decay.       

2. The Gaussian decay~\eqref{eq:cne2} of the dual window of
$\mathcal{G}(H_n, \Lambda )$ is usually expressed by saying that
$\gamma _n$ belongs to the Gelfand-Shilov space $S^{1/2}_{1/2}$. This
is considered the smallest reasonable space of test functions in
time-frequency analysis. An important consequence of this observation
is the existence of universal Gabor frames, i.e., the frame
$\mathcal{G}(H_n, \Lambda )$  is a (Banach) frame for all reasonable
modulation spaces. We refer to~\cite{gro07c} for a precise statement
and further details.

\vspace{ 3mm}

Finally let us briefly describe how the lower frame bound of the  Gabor
frame $\mathcal{G}(H_0, \Lambda )$ with Gaussian window behaves, when
the  lattice approaches  the critical size $s(\Lambda ) =1$.
Since the constants in Lemma~\ref{norms} depend on (the excentricity
of) the lattice $\Lambda $,   we fix a lattice $\Lambda $ with size
$s(\Lambda )=1$ and study the 
behavior of lower frame bound of $\mathcal{G}(H_0, q \Lambda )$,
as $q $ tends to $1$. As long as $q <1$,
$\mathcal{G}(H_0, q \Lambda )$ is a frame by the classical results
in~\cite{lyu,sw}, however, when $q  =1$, then
$\mathcal{G}(H_0, q \Lambda )$ cannot be a frame by the Balian-Low
theorem~\cite{BHW95}.  The upper frame bound can be controlled
uniformly with Proposition~\ref{bessel}, therefore    the lower frame
bound must $A$ 
converge to $0$ as $q \to 1$. 

\begin{prop}
  Assume that $s(\Lambda )=1$ and $q   <1$. 
   Let  $A_q$ denote the optimal lower
  frame bound of   $\mathcal{G}(H_0,q \Lambda )$. Then 
$$
A_q \geq c (1-q ^2)^2 = c (1-s(q \Lambda ))^2
$$ 
for some constant independent of $q $. 
\end{prop}

\begin{proof}
Let $\gamma _0$ be the dual window defined by ~\eqref{eq:dualcc}.  
 Corollary~\ref{lowerbound}  of  the
  Wexler-Raz relations  and Lemma~\ref{norms} imply that 
  $$
A_q \geq (n(q \Lambda ) \|\gamma _0\|_{M^1}^2)\inv \geq 
 \Big(16 n(q \Lambda) c(\overline{(q \Lambda )^\circ})^2 \Big)\inv
 (1-s(q \Lambda ))^2
$$
Thus we need to show that the two constants $n$ and $c$ are bounded, 
as long as $q $ is bounded away from $0$, $q \geq 1/2$, say.
Clearly the constant $n(q \Lambda )$, which measures the maximal number
of lattice points in a unit cube (Lemma~\ref{bessel}), can be bounded uniformly. 
As for $c$, which is the supremum of the $\sigma $-function over the
fundamental parallelogram~\eqref{C}, we note that 
$(q \Lambda )^\circ= q \inv \Lambda ^\circ$ 
and that 
$\sigma _{q \inv  \Lambda }(z) = q\inv \sigma _{\Lambda} (q z)$ by
\eqref{sigma}, \eqref{legendre} and \eqref{sconstants}. Consequently,
$\sup _{1/2\leq q \leq 1} c(q\inv \overline{\Lambda ^\circ})
< \infty $, and we are done. 
\end{proof}

\section{Appendix}

In this appendix we show how the structural results about Gabor
superframes can be derived from the corresponding well-known results
for scalar Gabor frames. 

\begin{proof}[Proof of Janssen's representation, Theorem~\ref{janssen_prop}]
  We look at the $l$-th component of $S\bof$. Using definition 
\eqref {foperator} it can be written explicitly as
 \[
 (S\bof)_l= \sum_{\lambda \in \Lambda} \lng \bof, \pi_\lambda \bog \rng
                               \pi_\lambda \bga _l =
  \sum_{j=1}^n    \sum_{\lambda \in \Lambda}  \lng f_j, \pi_\lambda g_j \rng
                               \pi_\lambda \gamma_l =
\sum_{j=1}^n S^\Lambda_{g_j,\gamma_l}f_j.
\]                               
If all $g_j$, $\gamma_j$ are in $M^1$, then Janssen's representation~\eqref{janssen}
can be applied to each of the frame-type operators occurring in the
above sum and we obtain that  
\beq
(S\bof)_l= s(\Lambda)^{-1}\sum_{\mu \in \Lambda^\circ} \sum_{j=1}^n
\lng \gamma_l, \pi_\mu g_j\rng \pi_\mu f_j= 
 \sum_{\mu \in \Lambda^\circ} \big ( \Gamma(\mu) \pi_\mu \bof \big )_l,               
\eeq
as was to be shown.
\end{proof}

\smallskip

\begin{proof}[Proof of the Wexler-Raz biorthonality conditions,
  Prop.~\ref{wexler_raz}] 
 We remark first that \tfs s on a lattice are linearly
independent in the following sense:  if $c=(c_\mu )_{\mu \in \Lambda ^\circ } \in
\ell ^\infty (\Lambda ^\circ )$ and $\sum _{\mu \in \Lambda ^\circ}
c_m \pi (\mu ) = 0$ as an operator from $M^1$ to $(M^1)^* $, then
$c_\mu = 0$ for all $\mu \in \Lambda ^\circ $. See, e.g., \cite{gro07a,
rieffel88}.  
If $S_{\mathbf{g}, \mathbf{\gamma}} = \mathrm{I}$, then by Janssen's
representation 
$$
f_{\ell } = (S\mathbf{f})_{\ell} =  s(\Lambda)\inv \, \sum
       _{\mu \in \Lambda ^\circ } \sum _{j=1}^n \langle \gamma _{\ell } , \pi _\mu
       g_j\rangle \pi _\mu f_j \, ,
$$
and so the linear independence forces $ s(\Lambda)\inv \langle \gamma
_{\ell } , \pi _\mu 
       g_j\rangle = \delta _{\mu , 0} \delta _{j,\ell }$, or in short
       notation, 
$s(\Lambda )\inv \Gamma (\mu ) = \delta _{\mu ,0} \mathrm{I}$. 

Conversely, if the biorthogonality condition $s(\Lambda )\inv \Gamma
(\mu ) = \delta _{\mu ,0} \mathrm{I}$  holds, then obviously
$S_{\mathbf{g}, \mathbf{\gamma}} = \mathrm{I}$. 
\end{proof}
\medskip

\begin{proof}[Proof of the Ron-Shen duality, Theorem~\ref{ron-shen}]
Assume first that  $\cG (\mathbf{g}, \Lambda )$ is a frame for 
$L^2(\bR^d, \bC ^n)$. Then the
canonical dual is given by $\bga ^\circ  = S\inv \mathbf{g}$, and by
general frame theory $\bga  ^\circ $ satisfies $S_{\mathbf{g},
 \bga ^\circ} = \mathrm{I}$. Furthermore, $\cG (\bga ^\circ ,
\Lambda )$ is again a frame for 
$L^2(\bR^d, \bC^n)$,  in particular it is a Bessel
sequence. If 
\begin{equation}
  \label{eq:vv14}
  f= \sum _{j=1}^n \sum _{\mu \in  \Lambda ^\circ }
c_{j,\mu  } \pi _\mu g_j
\end{equation} for some sequence $(c_{j,\mu })
\in \ell ^2(\{1, \dots , n\}\times \Lambda )$, 
then by
the Bessel property of $\cG (g_j, \Lambda ^\circ)$ for each $j$, we obtain 
that 
$$
\|f \|_2^2 = \|\sum _{j=1}^n \sum _{\mu \in \Lambda ^\circ }
c_{j,\mu } \pi _\mu g_j \| \leq  B \sum _{j=1} ^n \sum
_{\mu  \in \Lambda ^\circ } |c_{j,\mu }|^2_2 = B \|c\|_2^2 \, .
$$
For the converse inequality we use the Wexler-Raz relations. If $f \in
\lrd $ is given by \eqref{eq:vv14}, then the coefficients are uniquely
determined by 
$$
c_{j,\mu } = \langle f, \pi _\mu \gamma _j \rangle  \, .
$$
Again, since $\cG (\gamma _j, \Lambda ^\circ )$ possesses the Bessel
property, we find that $\|c\|_2^2 \leq A' \|f\|_2^2$. Altogether we
have shown that the set $\bigcup _{j=1}^n \cG (g_j , \Lambda ^\circ )$ is
a Riesz sequence. 

Conversely, assume that  $\bigcup _{j=1}^n \cG (g_j , \Lambda ^\circ )$
is a Riesz sequence that generates a subspace $\cK \subseteq \lrd
$. Then there exists a biorthogonal basis $\{e_{j,\mu }\}$ contained
in $\cK$. By the invariance of the Gabor systems $\cG (g_j, \Lambda
^\circ)$, we find that the biorthogonal basis must be of the form
$e_{j,\mu } = \pi _\mu \gamma _j$ for some $\gamma _j \in \lrd $. By
the general properties of Riesz bases, $\cG (\gamma _j, \Lambda ^\circ
)$ is a Bessel sequence for each $j$, and, after some rescaling, the biorthogonality states
that $s(\Lambda )\inv \, \langle \pi (\mu ')\gamma _j, \pi _\mu g _\ell \rangle =
\delta _{\mu,\mu '} \delta _{j,\ell }$. According to the Wexler-Raz
relations, this implies that $\cG (\mathbf{g}, \Lambda )$ is a frame
for $L^2(\bR ^d, \bC ^n) $. 
\end{proof}


 \end{document}